\documentclass[11pt]{article}
\usepackage{graphics,graphicx,subfigure,epsfig}
\usepackage{amsfonts,amsmath,amssymb,amsthm}

\newcommand{\R}{{\mathbb R}}

\newcommand{\X}{\mathcal{X}}

\newcommand{\Hc}{\mathcal{H}}
\newcommand{\C}{\mathcal{C}}
\newcommand{\Lc}{\mathcal{L}}
\newcommand{{\E}}{\mathbb E}

\newcommand{\F}{\mathbb F}

\newcommand{\dto}{\mathrel{\raisebox{-0.2ex}{$\xrightarrow{\,\raisebox{-0.2em}{$\scriptstyle \mathcal{D}$}\;}$}}}

\newcommand{\cov}{\mathop{\rm cov}}
\renewcommand{\rho}{\varrho}

\newtheorem{proposition}{Proposition}[section]
\newtheorem{theorem}[proposition]{Theorem}
\newtheorem{corollary}[proposition]{Corollary}

\newtheorem{remark}[proposition]{Remark}
\newtheorem{example}[proposition]{Example}

\newenvironment{tightlistleft}[1]{%
    \list{{\textup{(\roman{enumi})}}}{\settowidth\labelwidth{{\textup{(#1)}}}
    \leftmargin 0pt \advance\leftmargin\labelsep \itemindent \parindent
    \parsep 0pt plus 1pt minus 1pt \topsep 3pt \itemsep 0pt
    \usecounter{enumi}}}{\endlist}

\begin{document}

\title{\bf Statistical Estimation of Composite Risk Functionals and Risk Optimization Problems
}


\author{Darinka Dentcheva\thanks{Stevens Institute of Technology, Hoboken, NJ 07030, USA; Email: darinka.dentcheva@stevens.edu}
\and
Spiridon Penev\thanks{The University of New South Wales, Sidney, 2052 NSW, Australia; Email: s.penev@unsw.edu.au}
               \and
Andrzej Ruszczy\'nski\thanks{Rutgers University, Piscataway, NJ 08854, USA; Email: rusz@rutgers.edu}
}


\maketitle

\begin{abstract}
We address the statistical estimation of composite functionals which may be nonlinear in the probability measure. Our study is motivated by the need to estimate coherent measures of risk, which become increasingly popular in finance, insurance, and other areas associated with optimization under uncertainty and risk. We establish central limit formulae for composite risk functionals. Furthermore, we discuss the asymptotic behavior of optimization problems whose objectives are composite risk functionals and we establish a central limit formula of their optimal values when an estimator of the risk functional is used. While the mathematical structures accommodate commonly used coherent measures of risk, they have more general character, which may be of independent interest.\vspace{2ex}\\
\emph{Keywords:} Risk Measures \and Composite Functionals \and Central Limit Theorem
\end{abstract}

\section{Introduction}

Increased interest in the analysis of coherent measures
is motivated by their application as mathematical models of risk quantification in finance and other areas.
This line of research leads to new mathematical  problems in convex analysis,
optimization and statistics. The uncertainty in risk assessment is
expressed mathematically as a functional of random variable, which may be nonlinear with respect to the probability measure.
Most frequently, the risk measures of interest in practice arise when we evaluate gains or losses depending on the choice $z$, which represents the control of a decision maker and random quantities, which may be summarized in a random vector $X$. More precisely, we are interested in the functional $f(z,X)$, which may be optimized under practically relevant restrictions on the decisions $z$. Most frequently, some moments of the random variable $Y= f(z,X)$ are evaluated. However, when models of risk are used, the existing theory of statistical estimation is not always applicable.

Our goal is to address the question of statistical estimation of composite functionals depending on random vectors and their moments. Additionally, we analyse the optimal values of such functionals, when they depend on finite-dimensional decisions within a deterministic compact set. The known coherent measures of risk can be cast in the structures considered here and we shall specialize our results to several classes of popular risk measures.  We emphasize however, that the results address composite functionals of more general structure with a potentially wider applicability.

Axiomatic definition of risk measures was first proposed in \cite{Kijima}.
The currently accepted definition of a coherent risk measure was introduced in \cite{ADEH} for finite probability spaces and was further extended to more general spaces in \cite{RuS,FS2004}. Given a probability space $(\varOmega,{\mathcal F},P)$, we consider the set of random variables, defined on it, which have finite $p$-th moments and denote it by $\Lc^p(\varOmega,{\mathcal F},P)$.  A \emph{coherent measure of risk} is a convex, monotonically increasing, and positively homogeneous functional $\rho: \Lc^p(\varOmega,{\mathcal F},P)\to \bar{R}$, which satisfies the translation equivariant property $\rho(Y+a)=\rho(Y)+ a$ for all $a\in\R$. Here $\bar{R} = \R\cup\{+\infty\}$ and we assume that $Y$ represent losses, i.e., smaller realizations are preferred.
Related concepts are introduced in \cite{RUZ02,FS}.

A measure of risk is called \emph{law-invari\-ant}, if it depends only on the distribution of the random variable, i.e.,
if $\rho (X)= \rho(Y)$ for all random variables $X,Y\in \Lc^p(\varOmega,{\mathcal F},P)$ having the same distribution.

A practically relevant law-invariant coherent measure of risk is
the \emph{mean--semideviation} of order $p\ge 1$ (see \cite{ORejor,ORmp}, \cite[s. 6.2.2]{SDR}),
defined in the following way:
\begin{equation}
\label{semideviation}
\rho(X) = \E[X] + \kappa \big\| (X-\E[X])_+\big\|_p =
\E[X] + \kappa \Big[\E\big[\big(\max\{0,X- \E[X]\}\big)^p\big]\Big]^{\frac{1}{p}},
\end{equation}
where $\kappa \in [0,1]$. Note the nonlinearity with respect to the probability measure in formula \eqref{semideviation}.

Another popular law-invariant coherent measure of risk is the \emph{Average Value at Risk} at level $\alpha\in (0,1]$
(see \cite{RU02,ORsiam}), which is defined as follows:
\begin{equation}
\label{AVAR-ex}
{\rm AVaR}_{\alpha}(X)  =  \frac{1}{\alpha}  \int_{1-\alpha}^1 F_X^{-1}(\beta)\;d\beta
 = \min_{\eta\in\R} \bigg\{\eta + \frac{1}{\alpha} \E[(X- \eta)_+]\bigg\}.
\end{equation}
Here, $F_X(\cdot)$ denotes the distribution function of $X$. The reader may consult, for example, \cite[Chapter 6]{SDR} and the references therein, for more detailed discussion of these risk measures and their representation.

The risk measure ${\rm AVaR}_{\alpha}(\cdot)$ plays a fundamental role as a building block in the description of every law-invariant coherent risk measure  via the \emph{Kusuoka representation}.
 The original  result is presented in \cite{Kusuoka} for
 risk measures defined on $\Lc^{\infty}(\varOmega,\F,P)$, with an atomless probability space.
 It states that for every law-invariant coherent risk measure $\rho(\cdot)$,  a convex set $\mathcal{M}\subset \mathcal{P}(0,1]$  exists such that for all $Z\in  \Lc^{\infty}(\varOmega,\F,P)$, it holds
\begin{equation}
\label{Kusuoka}
\rho(X) = \sup_{m \in \mathcal{M}}  \int_0^1 {\rm
AVaR}_{\alpha}(X)\;m(d\alpha).
\end{equation}
Here $\mathcal{P}(0,1]$ denotes the set of probability measures on the interval $(0,1]$.
This result is extended to the setting of $\Lc^p$ spaces with $p \in [1, \infty )$; see \cite{FrittelliRosazza-2005},
 \cite{PflugR}, \cite{PflugW}, \cite{SDR}, \cite{DePeRu}, and the references therein.

The extremal representation of ${\rm AVaR}_{\alpha}(X)$ on the right hand side of \eqref{AVAR-ex}  was used as a motivation in \cite{Krokhmal} to propose the following higher-moment coherent measures of risk:
\begin{equation}
\label{Krokhmal}
\rho (X)=\min_{\eta\in\R} \bigg\{ \eta+\frac{1}{\alpha}\|(X - \eta)_+\|_p\bigg\}, \quad p > 1.
\end{equation}
These risk measures are special cases of a more general family considered in \cite{ChLi}; they are also examples of
optimized certainty equivalents of \cite{Ben-Tal}.
In the paper \cite{DePeRu}, the explicit Kusuoka representation for the higher-order risk measures (\ref{Krokhmal}) was described by utilising duality theorems from \cite{rock-duality}. These risk measures are used for portfolio optimization in \cite{Krokhmal}, where their advantages in
in comparison to the classical mean-variance optimization model of Markowitz (\cite{M52,M87}) is demonstrated on examples.
The recent work \cite{MaPe} indicates that if such type of risk measure is used as a risk criterion in European option portfolio optimization, the time evolution of the portfolio is superior to the evolution of a portfolio optimized with respect to the AVaR risk or with respect to the mean-variance optimization model of Markowitz. Similar observations were recently made in \cite{Gulten}.

A connection of measures of risk to the utility theories is discussed in the literature. Many of the risk measures of interest can be expressed via optimization of the so-called optimized certainty equivalent \cite{Ben-Tal} for a suitable choice of the utility function. Relations of risk measures to rank-dependent utility functions are given in \cite{FS2004}. In \cite{DDAR}, it is established that coherent measures of risk are a numerical representation of certain preference relation defined on the space of bounded quantile functions.

In practical applications, we deal with samples and stochastic models of the underlying random quantities. Therefore, the questions pertaining to statistical estimation of the measures of risk are crucial to the proper use of law-invariant measures of risk.
Several measures of risk have an explicit formula, which can be used as a \emph{plug-in estimator}, with the original measure $P$ replaced by the empirical measure. The empirical quantile is a natural estimator of the Value at Risk. A natural empirical estimator of ${\rm AVaR}_{\alpha}(X)$ leads to the use of the $L$-\emph{statistic}  (see \cite{JZ2003,DP}).
Furthermore,  the Kusuoka representation, as well as the use of distortion functions in insurance has motivated the construction and analysis of empirical estimates of spectral measures of risk using $L$-{statistic}. We refer to
\cite{JZ2003,BJPZ2008,JZ2007,BZ2010,Tsukahara2013,BK2012} for more details on this approach.
Some risk measures, such as the tail risk measures of form \eqref{Krokhmal}, cannot be estimated via simple explicit formulae but are obtained as a solution of a convex optimization problem with convex constraints.  Although asymptotic behavior of optimal values of sample-based expected value models has been investigated before (see \cite[Ch. 8]{Romisch03}, \cite[Ch. 5]{SDR} and the references therein), the existing results do not address models with risk measures.

Our paper is organized as follows. Section \ref{s:estimation} contains the key result of our paper, which establishes a central limit formula for a composite risk functional. We provide a characterization of the limiting distribution of the empirical estimators for such functionals. Section \ref{s:optimalvalue}, contains a central limit formula for risk functionals, which are obtained as a the optimal value of composite functionals.
Section \ref{last} provides asymptotic analysis and central limit formulae for the optimal value of optimization problems which use measures of risk in their objective functions.
We pay special attention to some popular measures and we discuss several illustrative examples in Sections \ref{s:estimation},\ref{s:optimalvalue}, and \ref{last}. In Section 5, we perform a simple simulation study to assess the accuracy of our approximations. Section 6 concludes.

\section{Estimation of composite risk functionals}
\label{s:estimation}

In the first part of our paper, we focus on functionals of the following form:
\[
\rho(X) = \E\Big[f_1\Big(\E\big[f_2\big(\E[\;\cdots f_k(\E[f_{k+1}(X)],X)]\;\cdots,X\big)\big],X\Big)\Big],
\]
where $X$ is an $m$-dimensional random vector, $f_j:\R^{m_j} \times \R^m \to \R^{m_{j-1}}$, $j=1,\dots,k$, with $m_0=1$ and $f_{k+1}:\R^m\to\R^{m_k}$.
Let $\X\subset \R^m$ be the domain of the random variable $X$.
We denote the probability distribution of $X$ by $P$.

Given a sample $X_1,\dots,X_n$ of independent identically distributed observations, we consider the following plug-in empirical estimate of the value of $\rho$:
\begin{align*}
\rho^{(n)} = \sum_{i_0=1}^n\frac{1}{n}\Big[f_1\Big(\sum_{i_1=1}^n\frac{1}{n}\big[f_2\big(\sum_{i_2=1}^n\frac{1}{n}[&\;\cdots f_k(\sum_{i_k=1}^n\frac{1}{n}f_{k+1}(X_{i_k}),X_{i_{k-1}})]\\
&\;\cdots,X_{i_1}\big)\big],X_{i_0}\Big)\Big]
\end{align*}

Our construction is motivated by the aim to estimate coherent measures of risk from the family of mean--semideviations (\cite{ORejor,ORmp}).

\begin{example}[{\rm \textbf{Semideviations}}]
\label{e:semideviation}
Consider the functional \eqref{semideviation}
representing the mean--semideviation of order $p\ge 1$.
In this case, we have $k=2, m=1$, and
\begin{align*}
f_1(\eta_1,x) &= x + \kappa\eta_1^{\frac{1}{p}},\\
f_2(\eta_2,x) &= \big[\max\{0,x- \eta_2\}\big]^p,\\
f_3(x) &= x. \hspace{7cm}\blacktriangle
\end{align*}
\end{example}

In order to formulate the main theorem of this section, we introduce several relevant quantities.
 We define:
\begin{align*}
\bar{f}_j(\eta_j)  &=   \int_{\X} f_j(\eta_j,x)\,P(dx),\quad j=1,\dots,k, \\
\mu_{k+1} &= \int_{\X} f_{k+1}(x)\,P(dx),\\
\mu_j &= \bar{f}_{j}(\mu_{j+1}),\quad j=1,\dots,k.
\end{align*}
 Suppose $I_j$ be compact subsets of $\R^{m_j}$ such that $\mu_{j+1}\in \text{int}(I_j)$, $j=1,\dots,k$.
We introduce the notation $\Hc = \C_1(I_1)\times \C_{m_1}(I_2)\times \dots \C_{m_{k-1}}(I_k)\times  \R^{m_{k}}$,
where $\C_{m_{j-1}}(I_j)$ is the space of continuous functions on $I_j$ with values in $\R^{m_{j-1}}$ equipped with the usually
supremum norm. The space $\R^{m_k}$ is equipped with the Euclidean norm and $\Hc$ is assumed equipped with the product norm.
We use Hadamard directional derivatives of the functions $f_{j}\big(\cdot,x)$ at points $\mu_{j+1}$ in directions
$\zeta_{j+1}$, i. e.,
\[
f'_{j}\big(\mu_{j+1},x;\zeta_{j+1}) = \lim_{{t\downarrow 0}\atop{s\to \zeta_{j+1}}}\frac{1}{t}
\big[ f_{j}\big(\mu_{j+1}+ts,x) - f_{j}\big(\mu_{j+1},x)\big].
\]
For every direction $d = (d_1,\dots,d_{k},d_{k+1}) \in\Hc$, we define recursively the sequence of vectors:
\begin{equation}
\label{recursive0}
\begin{gathered}
\xi_{k+1}(d) = d_{k+1},\\
\xi_{j}(d) = \int_{\X} f'_{j}\big(\mu_{j+1},x;\xi_{j+1}(d)\big)\,P(dx) + d_{j}\big(\mu_{j+1}\big),\quad j=k,k-1,\dots,1.
\end{gathered}
\end{equation}

\begin{theorem}
\label{t:main}
Suppose the following conditions are satisfied:
\begin{tightlistleft}{iv}
\item[(i)]  $\int \| f_j(\eta_j,x)\|^2 \;P(dx)<\infty$ for all $\eta_j\in I_j$, and $\int \| f_{k+1}(x)\|^2 P(dx)<\infty$;
\item[(ii)]  For all $x\in \X$, the functions $f_j(\cdot,x)$, $j=1,\dots,k$, are Lipschitz continuous:
\[
\|f_j(\eta_j',x)- f_j(\eta_j'',x)\| \le \gamma_j(x) \|\eta_j'-\eta_j''\|,\quad \forall\; \eta_j',\eta_j''\in I_j,
\]
and $\int \gamma_j^2(x)\;P(dx) <\infty$.
\item[(iii)]  For all $x\in \X$, the functions $f_j(\cdot,x)$, $j=1,\dots,k$, are Hadamard directionally differentiable.
\end{tightlistleft}
Then
\[
\sqrt{n}\big[ \rho^{(n)} - \rho\big] \dto \xi_1(W),
\]
where $W(\cdot) = \big( W_1(\cdot), \dots , W_{k}(\cdot), W_{k+1}\big) $ is a zero-mean Brownian process on
$I = I_1\times I_2 \times \dots \times I_k$. Here  $W_j(\cdot)$  is a  Brownian process
of dimension $m_{j-1}$ on $I_j$, $j=1,\dots,k$,  and  $W_{k+1}$ is an $m_k$-dimensional normal vector. The
covariance function of $W$ has the following form:
\begin{equation}
\label{covariance0}
\begin{aligned}
&\cov \big[W_i(\eta_i), W_j(\eta_j) \big] =
\int_{\X} \big[ f_i(\eta_i,x) - \bar{f}_i(\eta_i)\big] \big[ f_j(\eta_j,x) - \bar{f}_j(\eta_j)\big]^\top \;P(dx),\\
& \hspace{20em} \eta_i\in I_i,\ \eta_j\in I_j,\ i,j=1,\dots,k,\\
&\cov \big[ W_i(\eta_i) , W_{k+1} \big] =
\int_{\X} \big[ f_i(\eta_i,x) - \bar{f}_i(\eta_i)\big] \big[ f_{k+1}(x) - \mu_{k+1}\big]^\top \;P(dx),\\
& \hspace{20em} \eta_i\in I_i,\ i=1,\dots,k,\\
&\cov \big[ W_{k+1}, W_{k+1} \big] =
\int_{\X} \big[ f_{k+1}(x) - \mu_{k+1}\big] \big[ f_{k+1}(x) - \mu_{k+1}\big]^\top \;P(dx).
\end{aligned}
\end{equation}
\end{theorem}
\begin{proof}
We define $I = I_1\times I_2 \times \dots \times I_k$, $M= m_0+m_1+\dots+m_k$, and the vector-valued function
$
f: I \times \X \to \R^{M}
$
with block coordinates
$f_j(\eta_j,x)$, $j=1,\dots,k$, and $f_{k+1}(x)$. Similarly, we define $\bar{f}: I  \to \R^{M}$ with block
coordinates $\bar{f}_j(\eta_j)$, $j=1,\dots,k$, and $\mu_{k+1}$.
Consider the
empirical estimates of the function $\bar{f}(\eta)$:
\[
h^{(n)}(\eta)  =    \frac{1}{n}\sum_{i=1}^n f(\eta,X_i),\quad n =1,2,\dots.
\]
Due to assumptions (i)--(ii), all  functions $h^{(n)}$ are elements of the space
$\Hc$.

Furthermore, assumptions (i)--(ii) guarantee that the class of functions $f(\eta,\cdot)$, $\eta\in I$, is \emph{Donsker},
that is, the following uniform  Central Limit Theorem holds (see \cite[Ex. 19.7]{vdV}):
\begin{equation}
\label{Donsker}
\sqrt{n}\big( h^{(n)} - \bar{f} \big) \dto W,
\end{equation}
where $W$ is a zero-mean Brownian process on $I$ with covariance function
\begin{equation}
\label{covariance}
\cov \big[ W(\eta'), W(\eta'') \big]
=  \int_{\X} \big[ f(\eta',x) - \bar{f}(\eta')\big] \big[ f(\eta'',x) - \bar{f}(\eta'')\big]^\top \;P(dx).
\end{equation}
This fact will allow us to establish asymptotic properties of the sequence $\big\{\rho^{(n)}\big\}$.

First, we define a subset $H$ of $\Hc$  containing all elements $(h_1,\dots,h_k, h_{k+1})$ for which
$h_{j+1}(h_{j+2}(\cdots h_k(h_{k+1}) \cdots )) \in I_j$, $j=1,\dots,k$.
We define an operator $\Psi: H \to \R$ as follows
\[
\varPsi(h) = h_1\Big(h_2\big(\;\cdots h_k(h_{k+1})\;\cdots\big)\Big).
\]
By construction the value of $\rho(X)$ is equal to the value of  $\varPsi\big(\bar{f}\big)$ and the value of $\rho^{(n)}$ is equal to the value of $\varPsi\big(h^{(n)}\big)$.

To derive the limit properties of the sequence $\big\{ \rho^{(n)} \big\}$ we shall use \emph{Delta Theorem} (see, \cite{Romisch}).
The essence of applying the theorem is in identifying conditions under which a statement about a limit result related to convergence in distribution of a
scaled version of a statistic $h^{(n)},$  can be translated into a statement about a convergence in distribution of a scaled version of a transformed statistic $\Psi(h^{(n)}).$

To this end, we have to verify Hadamard directional differentiability of $\varPsi(\cdot)$ at $\bar{f}$.

Observe that the point $\bar{f}$ is an element of $H$, because $\mu_{j+1} \in \text{int}(I_j)$, $j=1,\dots,k$. Moreover,
due to assumption (ii), the following inequality is true for every $j=1,\dots,k$:
\begin{align*}
\lefteqn{\| h_{j}(h_{j+1}(h_{j+2}(\cdots h_k(h_{k+1}) \cdots ))) - \mu_j \|} \\
& \le
\| h_j - \bar{f_j}\| + \| \bar{f}_{j}(h_{j+1}(h_{j+2}(\cdots h_k(h_{k+1}) \cdots ))) - \bar{f}_j(\mu_{j+1}) \| \\
& \le
\| h_j - \bar{f_j}\| + \int \gamma_j(x)\;P(dx) \cdot  \| h_{j+1}(h_{j+2}(\cdots h_k(h_{k+1}) \cdots )) - \mu_{j+1} \|.
\end{align*}
Recursive application of this inequality demonstrates that $\bar{f}$
is an interior point of~$H$. Therefore, the quotients appearing in the definition of the Hadamard directional derivative
are well defined.

Conditions (ii) and (iii) imply that the functions $\bar{f}(\cdot)$ and $h^{(n)}(\cdot)$ are also Hadamard directionally differentiable.
Consider the operator $\varPsi_k(h) = h_k(h_{k+1})$ at $h \in \text{int}(H)$. Let $d^\ell=(d_1^\ell,\dots,d_k^\ell,d_{k+1}^\ell)\in \Hc$ be a sequence of directions converging in norm to an arbitrary direction $d \in \Hc$, when $\ell\to\infty$. For a sequence
$t_\ell\downarrow 0$ and $\ell$ sufficiently large, we have
\begin{align*}
\varPsi_k'(h;d)
&= \lim_{\ell\to\infty}\frac{1}{t_\ell}\big[\varPsi_k(h_k+t_\ell d_k^\ell,h_{k+1}+t_\ell d_{k+1}^\ell) -  \varPsi_k(h_k,h_{k+1})\big] \\
&= \lim_{\ell\to\infty} \frac{1}{t_\ell}\big([h_k+t_\ell d_k^\ell](h_{k+1}+t_\ell d_{k+1}^\ell) - h_k(h_{k+1})\big)\\
&= \lim_{\ell\to\infty} \frac{1}{t_\ell}\big(h_k(h_{k+1}+t_\ell d_{k+1}^\ell) - h_k(h_{k+1}) \big) + d_{k}^\ell(h_{k+1}+t_\ell d_{k+1}^\ell) \\
&= h_k'(h_{k+1}; d_{k+1}) + d_{k}(h_{k+1}).
\end{align*}
Consider now the operator
$\varPsi_{k-1}(h) = h_{k-1}\big(h_k(h_{k+1})\big) = h_{k-1}\big(\varPsi_k(h)\big)$.
By the chain rule for Hadamard directional derivatives we obtain
\begin{align*}
\varPsi_{k-1}'(h;d)
&= h'_{k-1}\big(\varPsi_k(h);\varPsi_k'(h;d)\big) + d_{k-1}\big(\varPsi_k(h)\big).
\end{align*}
In this way, we can recursively calculate the Hadamard directional derivatives of the operators
$\varPsi_j(h) = h_j\big(h_{j+1}(\,\cdots h_{k}(h_{k+1})\,\cdots)\big)$:
\begin{equation}
\label{recursive}
\varPsi_{j}'(h;d) = h'_{j}\big(\varPsi_{j+1}(h);\varPsi_{j+1}'(h;d)\big) + d_{j}\big(\varPsi_{j+1}(h)\big),\quad j=k,k-1,\dots,1.
\end{equation}
Now the Delta Theorem \cite{Romisch}, relation \eqref{Donsker}, and the Hadamard directional differentiability of $\varPsi(\cdot)$ at $\bar{f}$
imply that
\begin{equation}
\label{delta}
\sqrt{n}\big[ \rho^{(n)} - \rho(X) \big] =
\sqrt{n}\big[ \varPsi\big(h^{(n)}\big) - \varPsi\big(\bar{f}\big) \big]  \dto \varPsi'\big( \bar{f}, W \big).
\end{equation}
The application of the recursive procedure \eqref{recursive} at $h=\bar{f}$ and $d=W$ leads to formulae \eqref{recursive0}.
The covariance structure \eqref{covariance0} of $W$ follows directly from \eqref{covariance}.
\end{proof}

We return to Example~\ref{e:semideviation} and apply Theorem~\ref{t:main}.
\begin{example}[{\rm \textbf{Semideviations continued}}]
\label{e:2}
{\rm
We have defined the mappings
\begin{align*}
\bar{f}_1(\eta_1) &= \E[X] + \kappa\eta_1^{\frac{1}{p}}  = \int f_1(\eta_1,x) P(dx),\\
\bar{f}_2(\eta_2) &= \E \big\{ \big[\max\{0,X- \eta_2\}\big]^p \big\},
\intertext{and the constants}
\mu_3 &= \E[X],\quad \mu_2 = \E \big\{ \big[\max\{0,X- \E[X]\}\big]^p \big\},\quad \mu_1=\rho(X).
\end{align*}
We assume that $p>1$ and $I_2\subset\R$ is a compact interval containing the support of the random variable $X$.
The interval $I_1=[0,a]\subset\R$ can be defined by choosing $a$ so that $a\geq |X - \E(X)|^p$; for example $a$ may be equal to the diameter of the support of $X$ raised to power $p$.
The space $\Hc$ is $\C_1(I_1)\times \C_{2}(I_2)\times \R$  and we take a direction $d\in \Hc$.
Following \eqref{recursive0}, we calculate
\begin{align*}
\xi_2(d)
&= \bar{f}_2'(\mu_{3}; d_{3}) + d_{2}(\mu_{3}) = - p \E \big\{ \big[\max\{0,X- \mu_3\}\big]^{p-1} \big\}d_3 + d_2(\mu_3),\\
\xi_1(d) &= \bar{f}'_{1}\big(\mu_2;\xi_{2}(d)\big) + d_{1}\big(\mu_2\big) =  \frac{\kappa}{p}\mu_2^{\frac{1}{p} - 1}\xi_{2}(d) + d_{1}\big(\mu_2\big).
\end{align*}
We obtain the expression
\begin{multline}
\label{longformula}
\xi_1(W) =  W_1\big( \E \big\{ \big[\max\{0,X- \E[X]\}\big]^p \big\} \big) + {\ }  \\
 \frac{\kappa}{p}\Big( \E \big\{ \big[\max\{0,X- \E[X]\}\big]^p \big\} \Big)^{\frac{1-p}{p}} \times \\
\Big(W_2\big(\E[X]\big) - p  \E \big\{ \big[\max\{0,X- \E[X]\}\big]^{p-1} \big\}W_3\Big).
\end{multline}
The covariance structure of the process $W$ can be determined from \eqref{covariance0}. The process $W_1(\cdot)$ has the constant covariance function:
\begin{multline*}
\text{cov}\big[W_1(\eta'),W_1(\eta'')\big] \\
 =
\int_{\X} \big[ f_1(\eta',x) - \bar{f}_1(\eta')\big] \big[ f_1(\eta'',x) - \bar{f}_1(\eta'')\big] \;P(dx) =
\text{Var}[X].
\end{multline*}
It follows that $W_1(\cdot)$ has constant paths.
The third coordinate, $W_3$ has variance equal to $\text{Var}[X]$. It also follows from \eqref{covariance0}
that $\text{cov}\big[W_1(\eta),W_3\big] = \text{Var}[X]$.
Therefore, $W_1$ and $W_3$ are, in fact, one normal random variable, which we denote by $V_1$.

Observe that \eqref{longformula} involves only the value of the process $W_2$ at  $\mu_3=\E[X]$.
The variance of the random variable $V_2=W_2(\E[X])$ and its covariance with $V_1$ can be calculated from \eqref{covariance0} in a similar way:
\begin{align*}
&\text{Var}[V_2]  = \E\Big\{\Big( \big[\max\{0,X- \E[X]\}\big]^p - \E\big(\big[\max\{0,X- \E[X]\}\big]^p\big) \Big)^2 \Big\},\\
&\text{cov}[V_2,V_1] = {\ }\\
&\qquad \E\Big\{\Big( \big[\max\{0,X- \E[X]\}\big]^p - \E\big(\big[\max\{0,X- \E[X]\}\big]^p\big) \Big)
\Big( X-\E[X] \Big)\Big\}.
\end{align*}
Formula \eqref{longformula} becomes
\begin{multline}
\label{longformula1}
\xi_1(W) =  V_1 +
 \frac{\kappa}{p}\Big( \E \big\{ \big[\max\{0,X- \E[X]\}\big]^p \big\} \Big)^{\frac{1-p}{p}} \times \\
\Big(V_2 - p  \E \big\{ \big[\max\{0,X- \E[X]\}\big]^{p-1} \big\}V_1\Big).
\end{multline}
We conclude that
\[
\sqrt{n}\big[ \rho^{(n)} - \rho\big] \dto \mathcal{N}(0,\sigma^2),
\]
where the variance $\sigma^2$ can be calculated in a routine way as a variance of the right hand side of \eqref{longformula1},
by substituting the expressions for variances and covariances of $W_1$, $W_2$, and $W_3$. \hfill $\blacktriangle$
}
\end{example}

\begin{remark}
Following Example \ref{e:2}, we could derive the limiting distribution of $\sqrt{n}\big[ \rho^{(n)} - \rho\big]$
for $p=1$ as well. However, the risk measure for $p=1$ enjoys a simpler form and is already analysed in the literature (see, \cite[Section 6.5]{SDR}.)

\end{remark}


\section{Estimation of Risk Measures Representable as Optimal Value of Composite Functional}
\label{s:optimalvalue}

As an extension of the methods of section \ref{s:estimation},
we consider the following general setting. Functions $f_1:\R^d\times\R^s\to\R$, $f_2:\R^d\times\R^m\to\R^s$, and a random vector $X$ in $\R^m$ are given. Our intention is to estimate the value of a composite risk functional
\begin{equation}
\label{optimized}
\varrho = \min_{z\in Z}f_1\big(z,\E[f_2(z,X)]\big).
\end{equation}
where $Z\subset\R^d$ is a nonempty compact set.

We note that the compactness restriction is made for technical convenience and can be relaxed.

Let $X_1,\dots,X_n$ be a random iid sample from the probability distribution $P$ of $X$.
We construct the empirical estimate
\[
\rho^{(n)} =   \min_{z\in Z}f_1\Big(z, \textstyle{\frac{1}{n}\sum_{i=1}^n f_2(z,X_i)}\Big).
\]
Our intention is to analyze the asymptotic behavior of $\rho_n$, as $n\to\infty$.

Following the method of section \ref{s:estimation},
we define the mapping $\varPhi:Z\times \mathcal{C}(Z)\to \R$ as follows:
\[
\varPhi(z,h) = f_1\big(z,h(z)\big).
\]
The space $\R^d\times \mathcal{C}(Z)$ is equipped with the product norm of the euclidian norm on $\R^d$ and the supremum norm on
$\mathcal{C}(Z)$. We also define the functional $v:\mathcal{C}(Z)\to \R$,
\begin{equation}
\label{vh}
v(h) = \min_{z\in Z} \varPhi(z,h).
\end{equation}
Setting
\begin{align*}
\bar{h}(z)  &=  \E[f_2(z,X)],\\
h^{(n)}(z)  &=   \textstyle{\frac{1}{n}\sum_{i=1}^n f_2(z,X_i)},
\end{align*}
we see that
\begin{align*}
\varrho &= v(\bar{h}),\\
\varrho^{(n)} &= v(h^{(n)}),\quad n=1,2\dots.
\end{align*}

Let $\hat{Z}$ denote for the set of optimal solutions of problem \eqref{optimized}.
\begin{theorem}
\label{t:opt-est}
In addition to the general assumptions, suppose the following conditions are satisfied:
\begin{tightlistleft}{iii}
\item[(i)]  The function $f_2(z,\cdot)$ is  measurable for all $z\in Z$;
\item[(ii)] The function $f_1(z,\cdot)$ is  differentiable for all $z\in Z$,
and both $f_1(\cdot,\cdot)$ and its derivative with respect to the second argument, $\nabla f_1(\cdot,\cdot)$,
are continuous with respect to both arguments;
    \item[(iii)] An integrable function $\gamma(\cdot)$ exists such that
    \[
    \| f_2(z',x) - f_2(z'',x)\| \le \gamma(x)\| z' - z''\|
    \]
for all $z',z''\in Z$ and all $x\in \X$; moreover,
    $\int \gamma^2(x)\;P(dx) <\infty$.
\end{tightlistleft}
Then
\begin{equation}
\label{CLT-set}
\sqrt{n}\big[ \rho^{(n)} - \rho\big] \dto \min_{z\in \hat{Z}}
\big\langle \nabla f_1\big(z,\E[f_2(z,X)]\big),W(z)\big\rangle,
\end{equation}
where $W(z)$ is a zero-mean Brownian process on $Z$ with the covariance function
\begin{multline}
\label{covariance-opt}
\cov\big[ W(z'),W(z'')\big] =\\
\int_{\X} \big( f_2(z',x) - \E[f_2(z',X)]\big) \big( f_2(z'',x) - \E[f_2(z'',X)]\big)^\top \;P(dx).
\end{multline}
\end{theorem}
\begin{proof}
Observe that assumptions (i)-(ii) of Theorem \ref{t:main}
are satisfied due to the compactness of the set $Z$ and assumptions (ii)--(iii) of this theorem.
Therefore, formula \eqref{Donsker} holds:
\begin{equation}
\label{Donsker2}
\sqrt{n}\big( h^{(n)} - \bar{h} \big) \dto W. \notag
\end{equation}
The limiting process $W$ is a zero-mean Brownian process on $Z$ with covariance function \eqref{covariance-opt}.

Furthermore, due to assumption (ii), the function $\varPhi(\cdot,h)$ is continuous. As the set $Z$ is compact, problem \eqref{vh} has a nonempty solution set $S(h)$. By virtue of \cite[Theorem 4.13]{BonnansShapiro}, the optimal value function $v(\cdot)$ is Hadamard-directionally differentiable at $\bar{h}$ in every direction $d$  with
\[
v'(\bar{h};d) = \min_{z\in S(\bar{h})} \varPhi_h'(z,\bar{h})d,
\]
where $\varPhi'(z,h)$ is the Fr\'echet derivative of $\varPhi(z,\cdot)$ at $h$.
Therefore, we can apply the delta method (\cite{Romisch}) to infer that
\[
\sqrt{n}\big(v(h^{(n)}) -v(\bar{h})\big)\dto \min_{z\in S(\bar{h})} \varPhi_h'(z,\bar{h})W.
\]
Substituting the functional form of $\varPhi$, we obtain
 \[
 \varPhi_h'(z,\bar{h}) = \nabla f_1\big(z,\E[f_2(z,X)]\big)\delta_z,
 \]
 where $\delta_z$ is the Dirac measure at $z$. Application of this operator to the process $W$ yields formula \eqref{CLT-set}.
 Observe that $W(\cdot)$ has continuous paths and the minimum exists.
\end{proof}
\begin{corollary}
\label{c:cor1}
If, in addition to conditions of Theorem \ref{t:opt-est}, the set $\hat{Z}$ contains only one element $\hat{z}$, then the following
central limit formula holds:
\begin{equation}
\label{CLT-set-2}
\sqrt{n}\big[ \rho^{(n)} - \rho\big] \dto
\big\langle \nabla f_1\big(\hat{z},\E[f_2(\hat{z},X)]\big),W(\hat{z})\big\rangle,
\end{equation}
where $W(\hat{z})$ is a zero-mean normal vector with the covariance
\[
\cov\big[ W(\hat{z}),W(\hat{z})\big] = \cov\big[ f_2(\hat{z},X), f_2(\hat{z},X)\big].
\]
\end{corollary}

The following examples show that two notable categories of risk measures fall into the structure \eqref{optimized}
\begin{example}[{\rm \textbf{Average Value at Risk}}]
\label{e:AVAR}
{\rm
Average Value at Risk \eqref{AVAR-ex} is one of the most popular and most basic coherent measure of risk. Recall that
for a random variable $X$, it is representable as follows:
\[
{\rm AVaR}_{\alpha}(X)  = \min_{z\in\R} \bigg\{z + \frac{1}{\alpha} \E[(X- z)_+]\bigg\}.
\]
This measure fits in the structure \eqref{optimized} by setting
\begin{align*}
f_1(z,\eta) &= z+ \frac{1}{\alpha}\eta\\
f_2(z,X) &= \max (0, X-z).
\end{align*}
The plug-in empirical estimators of \eqref{AVAR-ex} have the following form
\[
\rho^{(n)} = \min_{z\in \R} \Big\{ z + \frac{1}{\alpha n} \sum_{i=1}^n \big(\max(0, X_i- z)\big) \Big\}.
\]
If the support of the distribution of $X$ is bounded, then so is the support of all empirical
distributions and we can assume that the $Z$ contains the support of the distribution.
Observe that all assumptions of Theorem~\ref{t:opt-est} are satisfied. If the distribution function of the random variable $X$ is continuous at $\alpha$, then the solution of the optimization problem at the right-hand side of \eqref{AVAR-ex} is unique. In that case, also the assumptions of Corollary \ref{c:cor1} are satisfied.
We conclude that
\[
\sqrt{n}\big[ \rho^{(n)} - \rho\big] \dto \frac{1}{\alpha} \Big(\E\big[\max(0,X-\hat{z}\big]\Big) W,
\]
where $W$ is a normal random variable with zero mean and variance
\[
\text{Var}[W] = \E \Big[ \Big( \max(0, X- \hat{z}) - \E\big[ \max(0, X- \hat{z} \big] \Big)^2 \Big].
\]
We note that the assumption of bounded support of the random variable $X$ is not really essential because, we could take a sufficiently large set $Z$, which would contain the corresponding quantile of the distribution function of $X$ and all empirical quantiles for sufficiently large sample sizes.

Additionally, we refer to another method for estimating the average value at risk at all levels simultaneously, which is discussed in \cite{DP}, where also central limit formulae under different set of assumptions are established.
$\hfill\blacktriangle$}
\end{example}

\begin{example}[{\rm \textbf{Higher-order Inverse Risk Measures}}]
\label{e:inverse}
{\rm
Consider a higher order inverse risk measure \eqref{Krokhmal} with $c=\frac{1}{\alpha}>1$:
\begin{equation}
\label{inverse}
\rho[X] = \min_{z\in \R} \Big\{ z + c \big\|\max(0, X- z)\big\|_p \Big\},
\end{equation}
where $p>1$ and $\|\cdot\|_p$ is the norm in the $\Lc^p$ space.
We  define:
\begin{align*}
f_1(z,y) &= z + c y^{\frac{1}{p}},\\
f_2(z,x) &= \big(\max(0, x- z)\big)^p.
\end{align*}
If the support of the distribution of $X$ is bounded, so is the support of all empirical
distributions. In this case, we can find a bounded set $Z$ (albeit larger than the support of $X$)
such that all solutions of problems \eqref{inverse} belong to this set.
For $p>1$ and $c>1$  problem~\eqref{inverse} has a unique solution, which we denote by $\hat{z}$.

The plug-in empirical estimators of \eqref{inverse} have the following form
\begin{equation}
\label{toillustrate}
\rho^{(n)} = \min_{z\in \R} \Big\{ z + c \Big(
\frac{1}{n} \sum_{i=1}^n \big(\max(0, X_i- z)\big)^p\Big)^{\frac{1}{p}} \Big\}.
\end{equation}
Observe that all assumptions of Theorem~\ref{t:opt-est} and Corollary \ref{c:cor1} are satisfied. We conclude that
\begin{equation}
\label{moreillustrate}
\sqrt{n}\big[ \rho^{(n)} - \rho\big] \dto \frac{c}{p} \Big(\E\big[\big(\max(0,X-\hat{z})\big)^p\big]\Big)^\frac{1-p}{p} W,
\end{equation}
where $W$ is a normal random variable with zero mean and variance
\[
\text{Var}[W] = \E \Big[ \Big( \big(\max(0, X- \hat{z})\big)^p - \E\big[ \big(\max(0, X- \hat{z})\big)^p\big] \Big)^2 \Big].
\]
$\hfill\blacktriangle$}
\end{example}

\section{Estimation of Optimized Composite Risk Functionals}
\label{last}

In this section, we are concerned with optimization problems in which the objective function is a composite risk functional. Our goal is to establish a central limit formula for the optimal value of such problems.

Our methods allow for the analysis of more complicated structures of optimized risk functionals:
\begin{equation}
\label{nested-optimized}
\varrho=\min_{u\in U}\E\Big[f_1\Big(u,\E\big[f_2\big(u,\E[\;\cdots f_k(u,\E[f_{k+1}(u,X)],X)]\;\cdots,X\big)\big],X\Big)\Big].
\end{equation}
Here $X$ is a $m$-dimensional random vector, $f_j:U\times\R^{m_j} \times \R^m \to \R^{m_{j-1}}$, $j=1,\dots,k$, with $m_0=1$ and $f_{k+1}:U\times\R^m\to\R^{m_k}$.
We assume that $U$ is a compact set in a finite dimensional space and
the optimal solution $\hat{u}$ of this problem is unique.


We define the functions:
\begin{align*}
\bar{f}_j(u,\eta_j)  &=   \int_{\X} f_j({u},\eta_j,x)\,P(dx),\quad j=1,\dots,k, \\
\bar{f}_{k+1}(u)  &=   \int_{\X} f_{k+1}({u},x)\,P(dx),\\
\intertext{and the quantities}
\mu_{k+1} &= \bar{f}_{k+1}(\hat{u}),\\
\mu_j &= \bar{f}_{j}(\hat{u},\mu_{j+1}),\quad j=1,\dots,k.
\end{align*}
We  assume that compact sets $I_1,\dots,I_{k}$ are selected
so that $\text{int}(I_k) \supset \bar{f}_{k+1}(U)$, and $\text{int}(I_j)\supset \bar{f}_{j+1}(U,I_{j+1})$, $j=1,\dots,k-1$.
Let us define the space
\[
\Hc=\C_1^{(0,1)}(U\times I_1)\times \C_{m_1}^{(0,1)}(U\times I_2)\times \dots
\C_{m_{k-1}}^{(0,1)}(U\times I_k)\times \C_{m_{k}}(U),
\]
where $\C_{m_{j-1}}^{(0,1)}(U\times I_j)$ is the space of $\R^{m_{j-1}}$-valued continuous functions on $U\times I_j$,
which are differentiable with respect to the second argument with continuous derivatives on $U\times I_j$.
We denote the Jacobian of $f_j({u},\eta_j,x)$ with respect to the second argument at  $\eta_j^{*}\in I_j$ by $f'_j(u,\eta_j^{*},x)$.
For every direction $d\in\Hc$, we define recursively the sequence of vectors:
\begin{equation}
\label{recursive2}
\begin{gathered}
\xi_{k+1}(d) = d_{k+1},\\
\xi_{j}(d) = \int_{\X}  f'_{j}(\hat{u},\mu_{j+1},x)\xi_{j+1}(d)\,P(dx) + d_{j}\big(\mu_{j+1}\big),\quad j=k,k-1,\dots,1.
\end{gathered}
\end{equation}
The empirical estimator is
\begin{align*}
\varrho^{(n)}=\min_{u\in U}\sum_{i=1}^n \frac{1}{n}\Big[f_1\Big(u,\sum_{i=1}^n \frac{1}{n}\big[f_2\big(u,\sum_{i=1}^n \frac{1}{n}[&\;\cdots f_k(u,\sum_{i=1}^n \frac{1}{n}[f_{k+1}(u,X)],X)]\\
&\;\cdots,X\big)\big],X\Big)\Big].
\end{align*}
We establish the following result.
\begin{theorem}
\label{t:main-opt}
Suppose the following conditions are satisfied:
\begin{tightlistleft}{iii}
\item[(i)]
$\int_\X \| f_j(u,\eta_j,x)\|^2 \;P(dx)<\infty$ for all $\eta_j\in I_j$, $u\in U$, $j=1,\dots,k$, and \break $\int_\X \| f_{k+1}(u,x)\|^2 P(dx)<\infty$
for all $u\in U$;
\item[(ii)]  The functions $f_j(\cdot,\cdot,x)$, $j=1,\dots,k$, and $f_{k+1}(\cdot,x)$ are Lipschitz continuous for every $x\in \X$:
\begin{align*}
\|f_j(u',\eta_j',x)- f_j(u'',\eta_j'',x)\| &\le \gamma_j(x)\big(\|u'-u''\|+ \|\eta_j'-\eta_j''\|\big),\quad j=1,\dots,k.\\
\|f_{k+1}(u',x)- f_{k+1}(u'',x)\| &\le \gamma_{k+1}(x)\|u'-u''\|,
\end{align*}
for all  $\eta_j',\eta_j''\in I_j$, $u',u''\in U$; moreover, $\int \gamma_j^2(x)\;P(dx) <\infty$, $j=1,\dots,k+1$;
\item[(iii)]  The functions $f_j(u,\cdot,x)$, $j=1,\dots,k$, are continuously differentiable
for every $x\in \X$, $u\in U$; moreover, their derivatives are continuous with respect to the first two arguments.
\end{tightlistleft}
Then
\[
\sqrt{n}\big[ \rho^{(n)} - \rho\big] \dto \xi_1(W),
\]
where $W(\cdot) = \big( W_1(\cdot), \dots , W_{k}(\cdot), W_{k+1}\big) $ is a zero-mean Brownian process on
$I = I_1\times I_2 \times \dots \times I_k$. Here  $W_j(\cdot)$  is a  Brownian process
of dimension $m_{j-1}$ on $I_j$, $j=1,\dots,k$,  and  $W_{k+1}$ is an $m_k$-dimensional normal vector.
The covariance function of $W(\cdot)$ has the following form
\begin{equation}
\label{covariance2}
\begin{aligned}
&\cov\big[ W_i(\eta_i), W_j(\eta_j) \big]
= {}\\
&\qquad \int_{\X} \big[ f_i(\hat{u},\eta_i,x) - \bar{f}_i(\hat{u},\eta_i)\big] \big[ f_j(\hat{u},\eta_j,x) - \bar{f}_j(\hat{u},\eta_j)\big]^\top \;P(dx),\\
&\qquad\qquad \eta_i\in I_i,\ \eta_j\in I_j,\ i,j=1,\dots,k\\
&\cov\big[ W_i(\eta_i), W_{k+1} \big]
= {}\\
&\qquad
\int_{\X} \big[ f_i(\hat{u},\eta_i,x) - \bar{f}_i(\hat{u},\eta_i)\big] \big[ f_{k+1}(\hat{u},x) - \bar{f}_{k+1}(\hat{u})\big]^\top \;P(dx),\\
&\qquad \qquad \eta_i\in I_i,\ i=1,\dots,k\\
&\cov\big[ W_{k+1}, W_{k+1} \big]
= {}\\
&\qquad \int_{\X} \big[ f_{k+1}(\hat{u},x) - \bar{f}_{k+1}(\hat{u})\big] \big[ f_{k+1}(\hat{u},x) - \bar{f}_{k+1}(\hat{u})\big]^\top \;P(dx).
\end{aligned}
\end{equation}
\end{theorem}
\begin{proof}
We follow the main line of argument of the proof of Theorem \ref{t:main}.
We define  $M= m_0+m_1+\dots+m_k$ and the vector-valued function
$
f: U\times I \times \X \to \R^{M}
$
with block coordinates
$f_j(u,\eta_j,x)$, $j=1,\dots,k$, and $f_{k+1}(u,x)$. Similarly, we define $\bar{f}: U\times I  \to \R^{M}$ with block
coordinates $\bar{f}_j(u,\eta_j)$, $j=1,\dots,k$, and $\bar{f}_{k+1}(u)$.
Consider the
empirical estimates of the function $\bar{f}(u,\eta)$:
\[
h^{(n)}(u,\eta)  =    \frac{1}{n}\sum_{i=1}^n f(u,\eta,X_i),\quad n =1,2,\dots.
\]
Due to our assumptions, for sufficiently large $n$ all these functions are elements of the space $\Hc$.

Owing to assumptions (i)--(ii),  the class of functions $f(u,\eta,\cdot)$, $u\in U$, $\eta\in I$, is \emph{Donsker},
that is  the following uniform  Central Limit Theorem holds (see \cite[Ex. 19.7]{vdV}):
\begin{equation}
\label{Donsker3}
\sqrt{n}\big( h^{(n)} - \bar{f} \big) \dto W,
\end{equation}
where $W$ is a zero-mean Brownian process on $U\times I$ with covariance function
\begin{multline}
\label{covariance3}
\cov\big[ W(u',\eta'), W(u'',\eta'')\big]
=  {\quad} \\
\int_{\X} \big[ f(u',\eta',x) - \bar{f}(u',\eta')\big] \big[ f(u'',\eta'',x) - \bar{f}(u'',\eta'')\big]^\top \;P(dx).
\end{multline}
This fact will allow us to establish asymptotic properties of the sequence $\big\{\rho^{(n)}\big\}$.
We define an operator $\Psi: \Hc \to \R$ as follows
\[
\varPsi(u,h) = h_1\Big(u,h_2\big(u,\;\cdots h_k(u,h_{k+1}(u))\;\cdots\big)\Big).
\]
By definition,
\begin{align*}
\rho(X) &= \min_{u\in U}\varPsi\big(u,\bar{f}\big),\\
\rho^{(n)} &= \min_{u\in U}\varPsi\big(u,h^{(n)}\big).
\end{align*}
To apply {Delta Theorem} to the sequence $\big\{ \rho^{(n)} \big\}$, we have to verify Hadamard directional differentiability of the optimal value function $v(\cdot) = \min_{u\in U}\varPsi(u,\cdot)$ at $\bar{f}$. Observe that our assumptions imply that
the conditions of \cite[Thm. 4.13]{BonnansShapiro} are satisfied. As the optimal solution set is a singleton,
the function $v(\cdot)$ is differentiable at $\bar{f}$ with the Fr\'echet derivative
\[
v'(\bar{f}) = \varPsi'(\hat{u},\bar{f}),
\]
where $\varPsi'(u,f)$ is the Fr\'echet derivative of $\varPsi(u,\cdot)$ at $f$.
The remaining derivations are identical as those in the proof of Theorem \ref{t:main}. We only need
substitute $\hat{u}$ as an additional argument of all functions involved.
\end{proof}

\begin{example}[{\rm \textbf{Optimization problems with mean--semideviation}}]
{\rm
Consider now an optimization problem involving a mean--semideviation measure of risk
\begin{equation}
\label{msd-min}
\min_{u\in U} \rho[\varphi(u,X)] = \E[\varphi(u,X)] + \kappa \Big(\E \big[ \big(\varphi(u,X) - \E[\varphi(u,X)]\big)_+^p\big]\Big)^\frac{1}{p},
\end{equation}
where $\varphi:\R^d \times \X \to \R$.
We have
\begin{align*}
{f}_1(\eta_1,u,x) &= \kappa\eta_1^{\frac{1}{p}} + \varphi(u,x),\\
{f}_2(\eta_2,u,x) &= \big\{ \big[\max\{0,\varphi(u,x)- \eta_2\}\big]^p \big\},\\
{f}_3(u,x) &= \varphi(u,x),
\end{align*}
and
\begin{align*}
\bar{f}_1(\eta_1,u) &= \kappa\eta_1^{\frac{1}{p}} + \E[\varphi(u,X)],\\
\bar{f}_2(\eta_2,u) &= \E \big\{ \big[\max\{0,\varphi(u,X)- \eta_2\}\big]^p \big\},\\
\bar{f}_3(u) &= \E[\varphi(u,X)].
\end{align*}
We assume that $p>1$.
Suppose $\hat{u}$ is the unique solution of problem \eqref{msd-min}.
We set $\mu_3 = \E[\varphi(\hat{u},X)]$.
Then  $\mu_2 = \E \big\{ \big[\max\{0,\varphi(\hat{u},X)- \E[\varphi(\hat{u},X)]\}\big]^p \big\}$ and $\mu_1=\rho(X)$.
Following \eqref{recursive2}, we calculate
\[
\xi_2(d)
= \bar{f}_2'(\mu_{3},\hat{u}; d_{3}) + d_{2}(\mu_{3}) = - p \E \big\{ \big[\max\{0,\varphi(\hat{u},X)- \mu_3\}\big]^{p-1} \big\}d_3 + d_2(\mu_3),
\]
\[
\xi_1(d) = \bar{f}'_{1}\big(\mu_2,\hat{u};\xi_{2}(d)\big) + d_{1}\big(\mu_2\big)
=  \frac{\kappa}{p}\mu_2^{\frac{1}{p} - 1}\xi_{2}(d) + d_{1}\big(\mu_2\big).
\]
We obtain the expression
\begin{multline}
\label{longformula2}
\varPsi_1'(\bar{f};W) = W_1\big( \E \big\{ \big[\max\{0,\varphi(\hat{u},X)- \E[\varphi(\hat{u},X)]\}\big]^p \big\} \big) + {\ }  \\
\qquad \frac{\kappa}{p}\Big( \E \big\{ \big[\max\{0,\varphi(\hat{u},X)- \E[\varphi(\hat{u},X)]\}\big]^p \big\} \Big)^{\frac{1-p}{p}} \times\\
\qquad \Big(W_2\big(\E[\varphi(\hat{u},X)]\big)-p  \E \big\{ \big[\max\{0,\varphi(\hat{u},X)- \E[\varphi(\hat{u},X)]\}\big]^{p-1} \big\}W_3\Big).
\end{multline}
The covariance structure of the process $W$ can be determined from \eqref{covariance3}, similar to Example \ref{e:2}.
The process $W_1(\cdot)$ has the constant covariance function:
\[
\text{cov}\big[W_1(\eta_1(\hat{u})),W_1(\eta_1(\hat{u}))\big]  =
\text{Var}[\varphi(\hat{u},X)].
\]
The third coordinate, $W_3$ has variance equal to $\text{Var}[\varphi(\hat{u},X)]$.
Also, $$\text{cov}(W_1(\eta_1(\hat{u})),W_3) = \text{Var}[\varphi(\hat{u},X)],$$ and thus
$W_1$ and $W_3$ have the same normal distribution and are perfectly correlated.

The variance function of $W_2(\cdot)$ and its covariance with $W_1$ (and $W_3$) can be calculated in a similar way:
\begin{multline*}
\text{Var}[W_2(\E[\varphi(\hat{u},X)])]  = \E\Big\{\Big( \big[\max\{0,\varphi(\hat{u},X)- \E[\varphi(\hat{u},X)]\}\big]^p - \\ \E\big(\big[\max\{0,\varphi(\hat{u},X)- \E[\varphi(\hat{u},X)]\}\big]^p\big) \Big)
\Big( \varphi(\hat{u},X)-\E[\varphi(\hat{u},X)] \Big)\Big\}.
\end{multline*}
We conclude that
\[
\sqrt{n}\big[ \rho^{(n)} - \rho\big] \dto \mathcal{N}(0,\sigma^2),
\]
where the variance $\sigma^2$ can be calculated in a routine way as a variance of the right hand side of \eqref{longformula2},
by substituting the expressions for variances and covariances of $W_1$, $W_2$, and $W_3$.
$\hfill\blacktriangle$
}
\end{example}
\section{A simulation study}
In this section we illustrate the convergence of some estimators discussed in this paper to the limiting  normal distribution. Many previously known results for the case $p=1$ have been investigated thoroughly in the literature (see, e.g., \cite{SRIRF}) and we will not dwell upon these here. We will only illustrate the case about Higher-order Inverse Risk Measures as discussed in Example 4 for the case $p>1.$
More specifically, we take independent identically distributed observations $X_i, i=1,2,\dots, n$ from an independent identically distributed $X\sim \mathcal{N}(0,3)$ observations. We take $\epsilon=0.05$ and $p=2.$
In that case $c=20.$ Numerical calculation in Matlab delivers the theoretical argument minimum $z^*=14.5048$ and the value of the risk in (\ref{inverse}) being $\rho [X]=15.5163.$ The standard deviation of the random variable in the right hand side of (\ref{moreillustrate}) is 16.032.  The plug-in estimator $\rho^{(n)}$ of this risk can be represented as a solution of a convex optimization problem with convex constraints and hence a unique solution can be found by any package that solves such type of problems. We have used the {\tt{cvx}} package that can be operated within {\tt{matlab}}.
\noindent Denoting $d_i=\max  (X_i-z,0)  , i=1,2,\dots,n$ and putting all $d_i, i=1,2,\dots,n$ in a vector $\textbf{d}$ we
can rewrite our optimization problem as follows:
\begin{equation}
\label{convconv}
\begin{aligned}
\min_{z,\textbf{d}} \ & \big\{ c\frac{1}{n^{1/p}}  (\sum_{i=1}^n  d_i^p ) ^{1/p} +z\big\} \\
\text{subject to}\ &
 X_i-z \le d_i, \ d_i\ge 0, \ i=1,2,\dots,n.
 \end{aligned}
\end{equation}
        The numerical solution to this optimization problem gives us the estimator  $\rho^{(n)}.$ To get an idea about the speed of convergence to the limiting distribution in (\ref{toillustrate}) we simulate $m=2500$ risk estimators $\rho^{(n)}_j, j=1,2,\dots,2500$ for a given sample size $n$ and draw their histogram. The number of bins for the histogram is determined by the rough ``squared root of the sample size" rule. This histogram is superimposed to the  $\mathcal{N}(15.5163, (16.032/\sqrt{n})^2)$ density. As $n$ is increased, our theory suggests that the histogram and the normal density graph will look more and more similar in shape. Their closeness indicates how quickly the central limit theorem pops up in this case.

\begin{figure}[htbp]
  \begin{center}
    \mbox{
      \subfigure[$n=1000$]{\scalebox{0.40}{\includegraphics{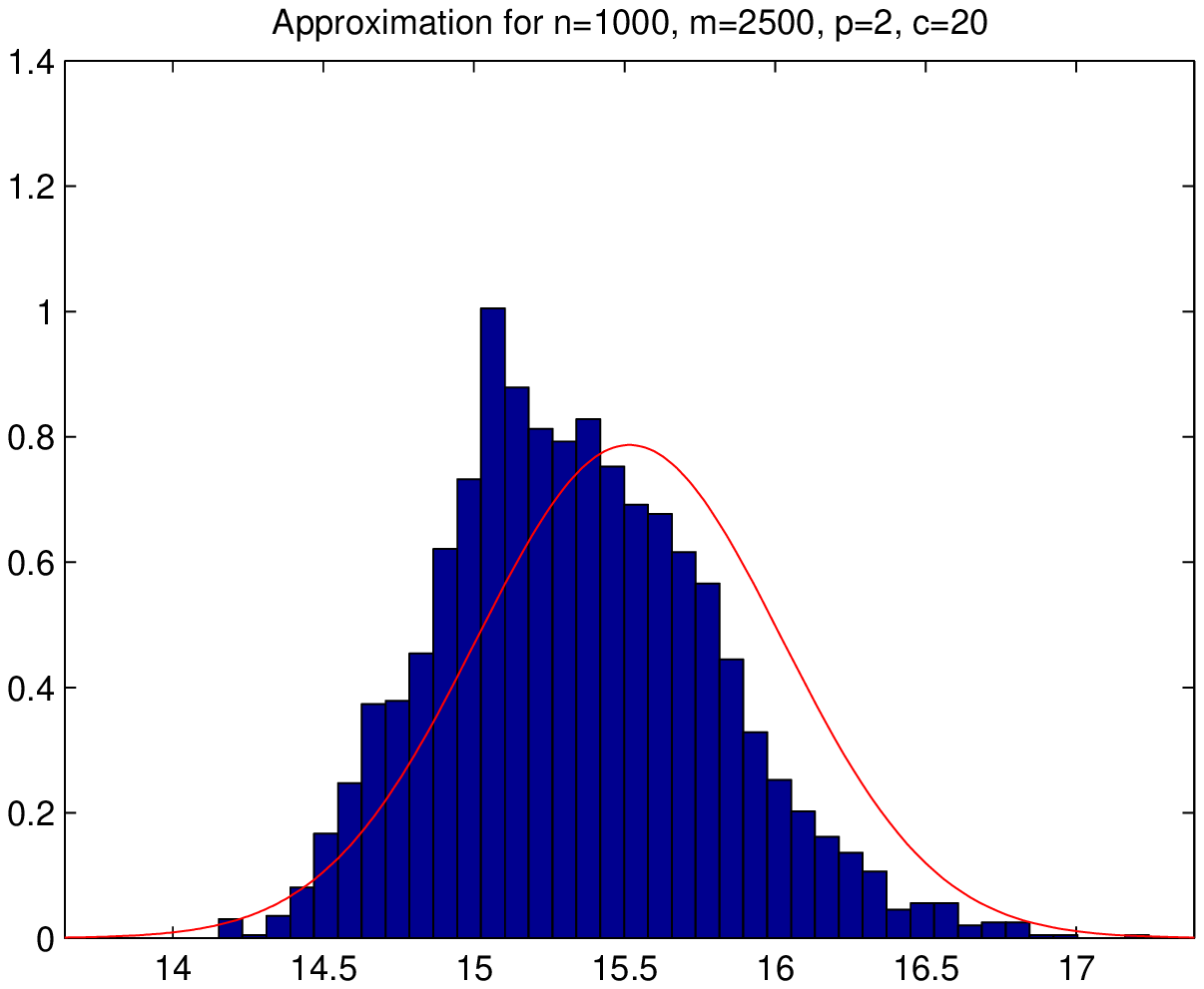}}}
      \subfigure[$n=2000$]{\scalebox{0.40}{\includegraphics{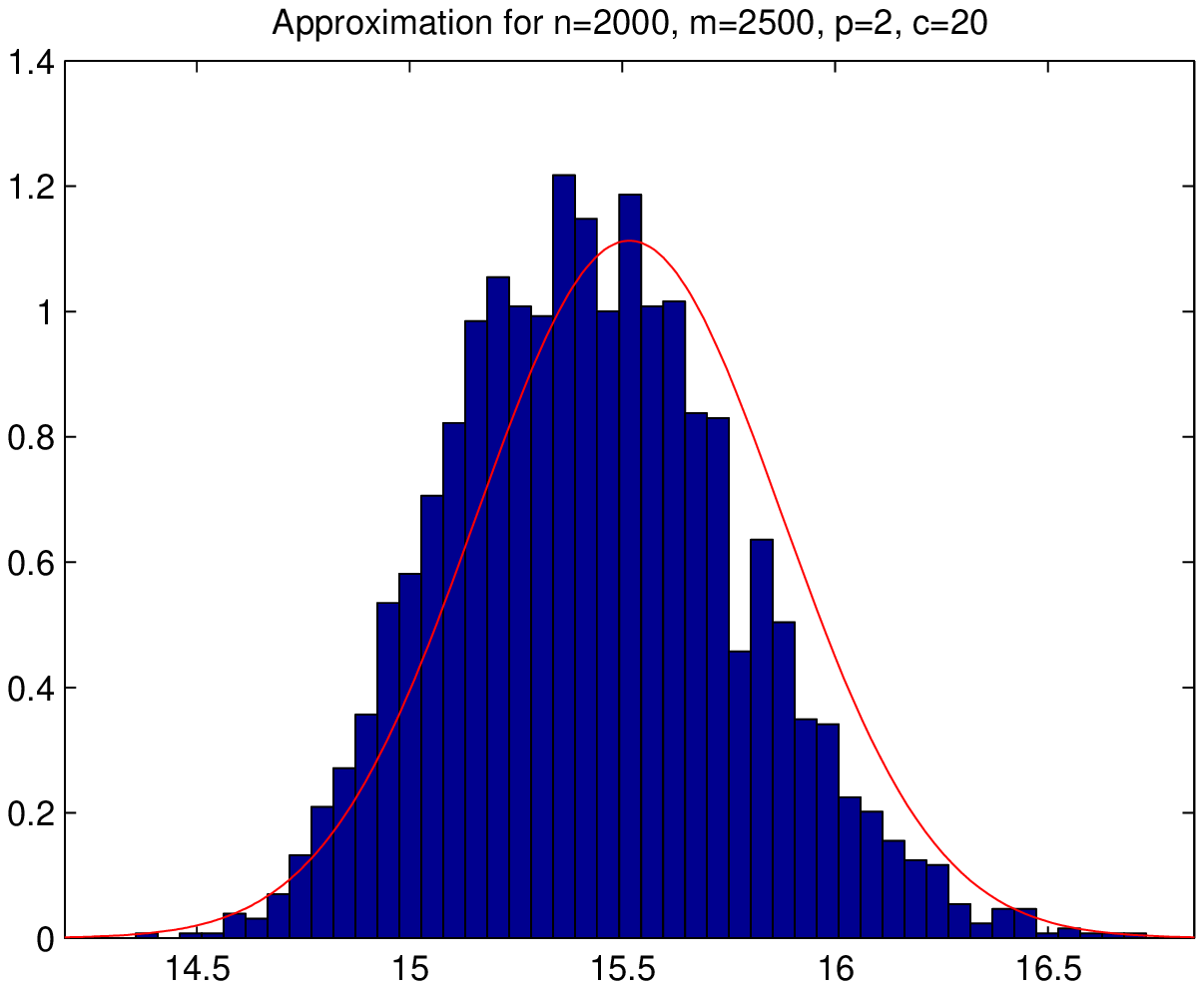}}}
      }
      \mbox{
      \subfigure[$n=4000$]{\scalebox{0.40}{\includegraphics{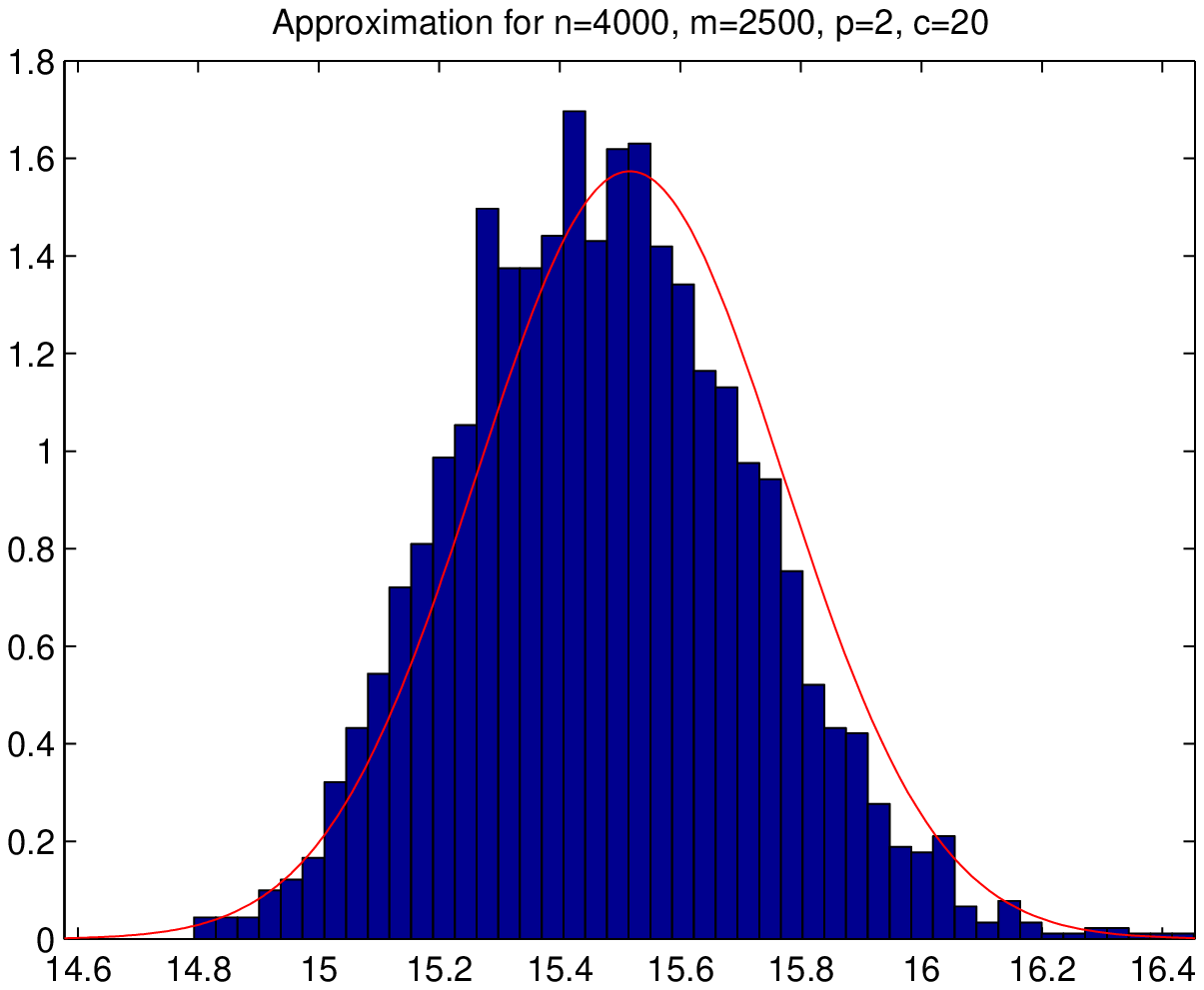}}}
      \subfigure[$n=8000$]{\scalebox{0.40}{\includegraphics{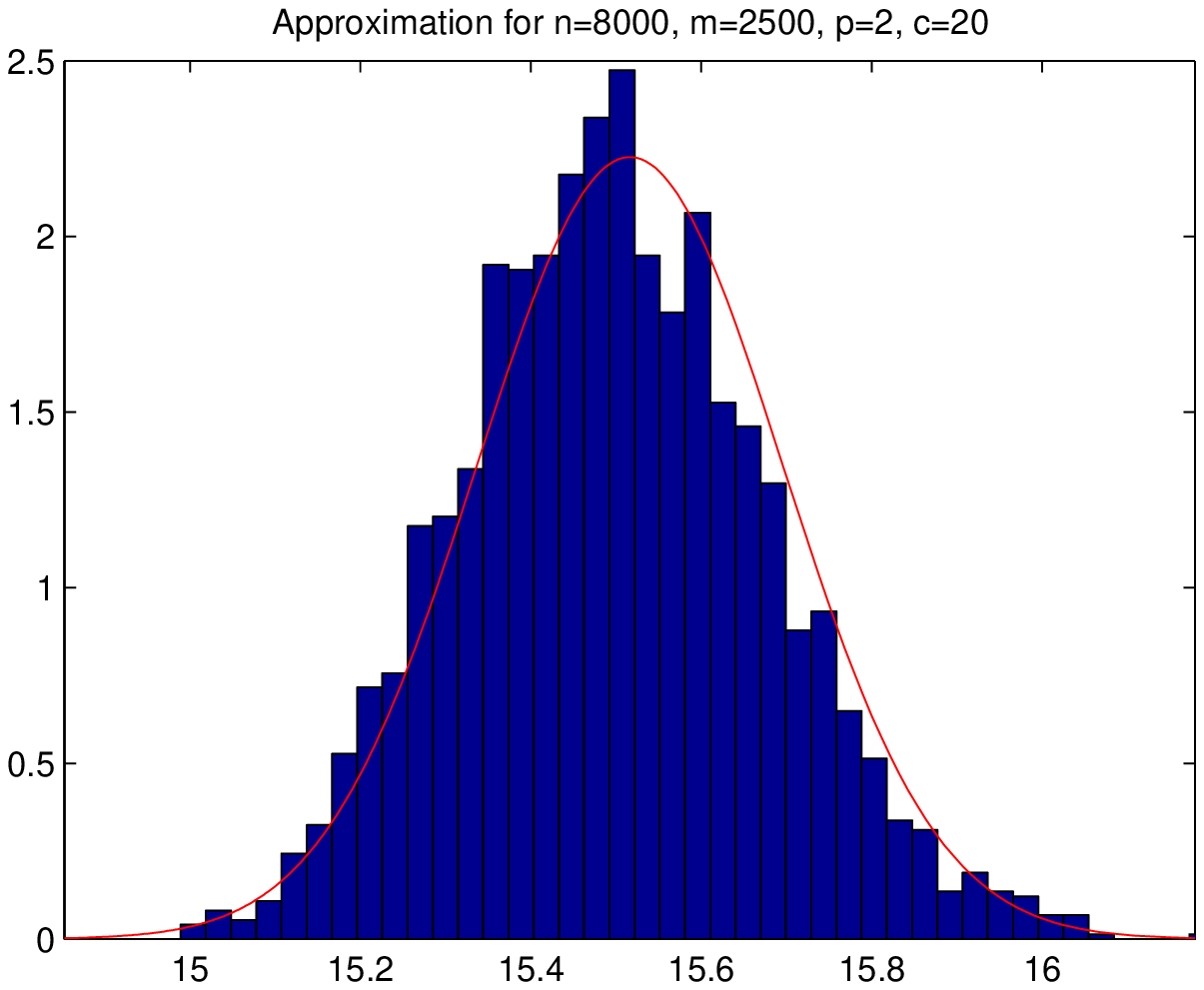}}}
      }
    \caption{\textit{Density histogram of the distribution of the estimator $\rho_n $ for increasing  values of $n$ and its normal approximation using Theorem 2 and $X\sim \mathcal{N}(10,3).$ }}
    \label{fig1}
  \end{center}
\end{figure}
\begin{figure}[htbp]
  \begin{center}
    \mbox{
      \subfigure[$df=60$]{\scalebox{0.40}{\includegraphics{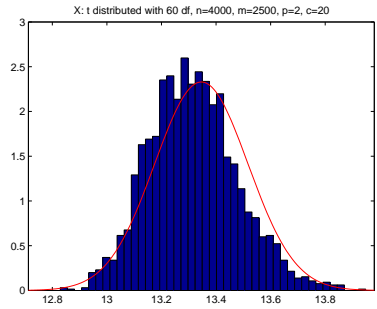}}}
      \subfigure[$df=8$]{\scalebox{0.40}{\includegraphics{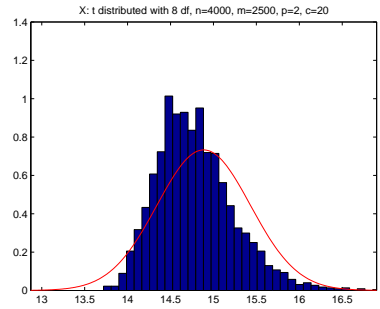}}}
      }
      \mbox{
      \subfigure[$df=6$]{\scalebox{0.40}{\includegraphics{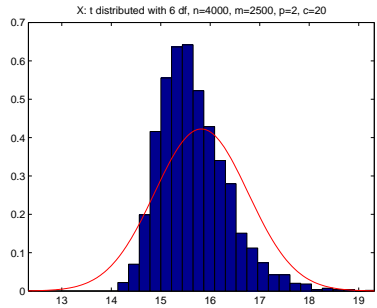}}}
      \subfigure[$df=4$]{\scalebox{0.40}{\includegraphics{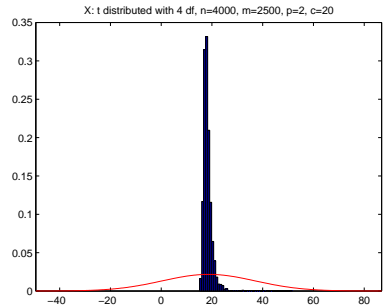}}}
      }
    \caption{\textit{Density histogram of the distribution of the estimator $\rho_n $ for  $n=4000$ and $X\sim t_{\nu}$ with $\nu$ being 60, 8, 6 and 4. }}
    \label{fig2}
  \end{center}
\end{figure}
\begin{figure}[htbp]
  \begin{center}
    \mbox{
      \subfigure[$p=1$]{\scalebox{0.40}{\includegraphics{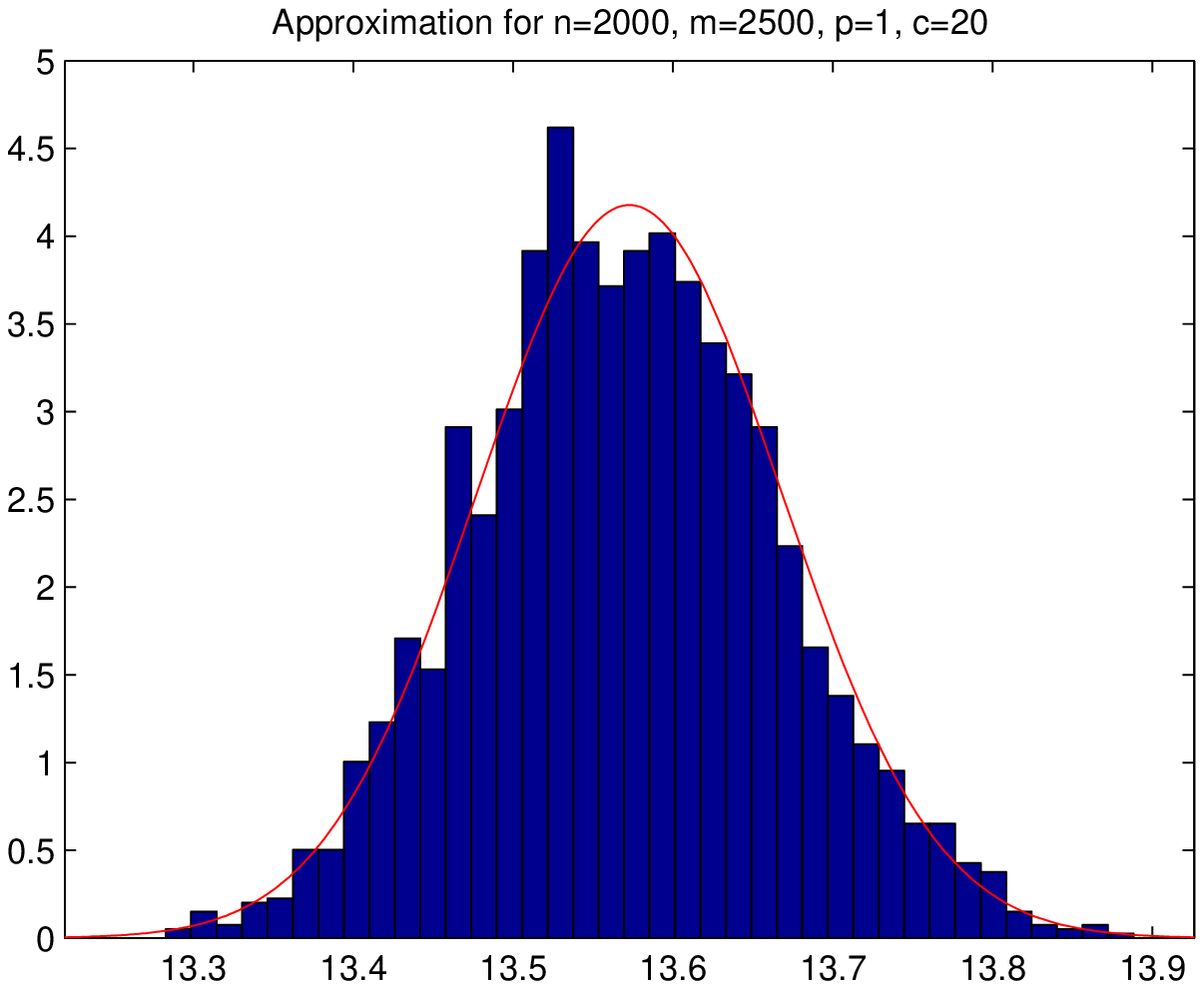}}}
      \subfigure[$p=1.5$]{\scalebox{0.40}{\includegraphics{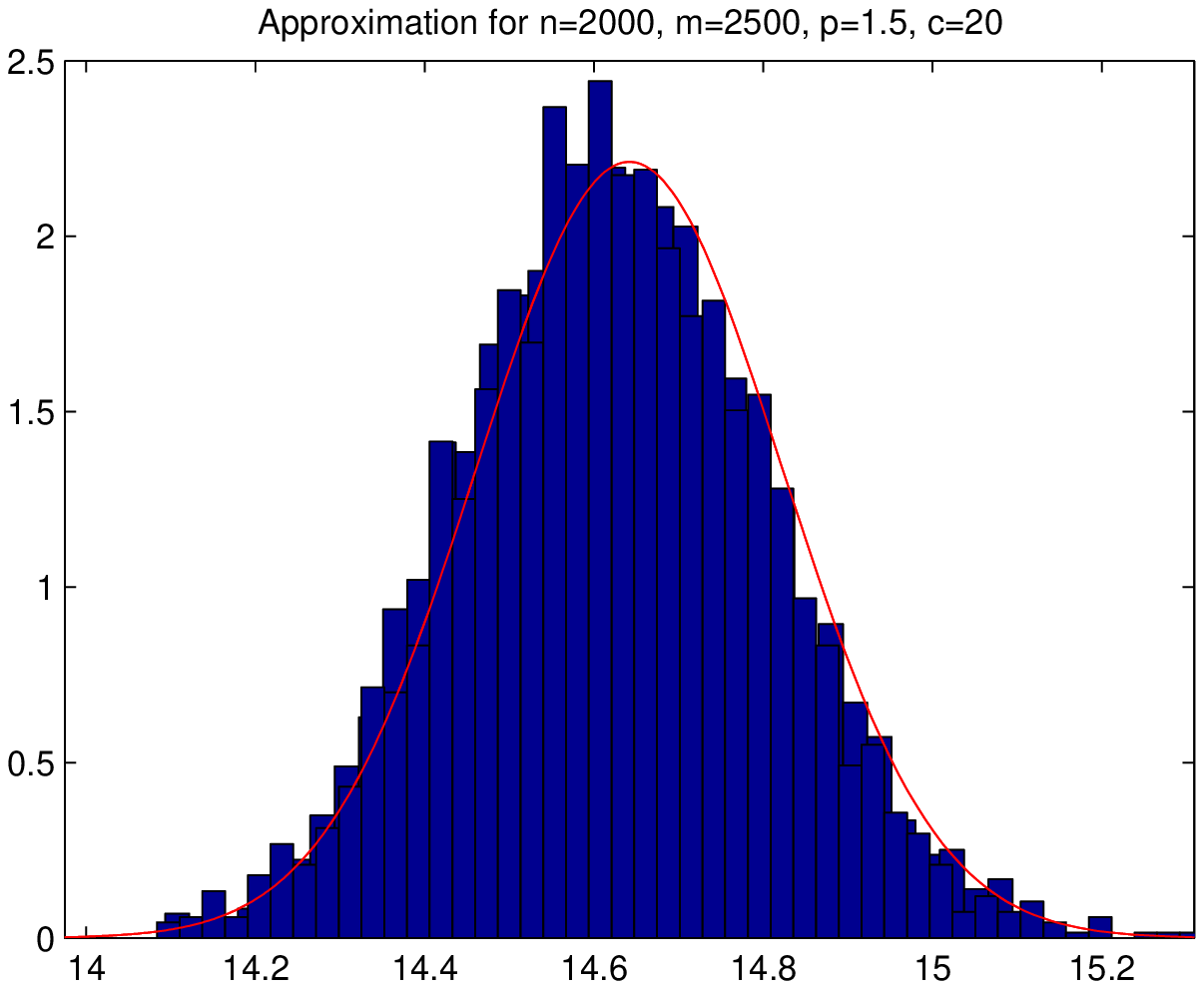}}}
      }
      \mbox{
      \subfigure[$p=2$]{\scalebox{0.40}{\includegraphics{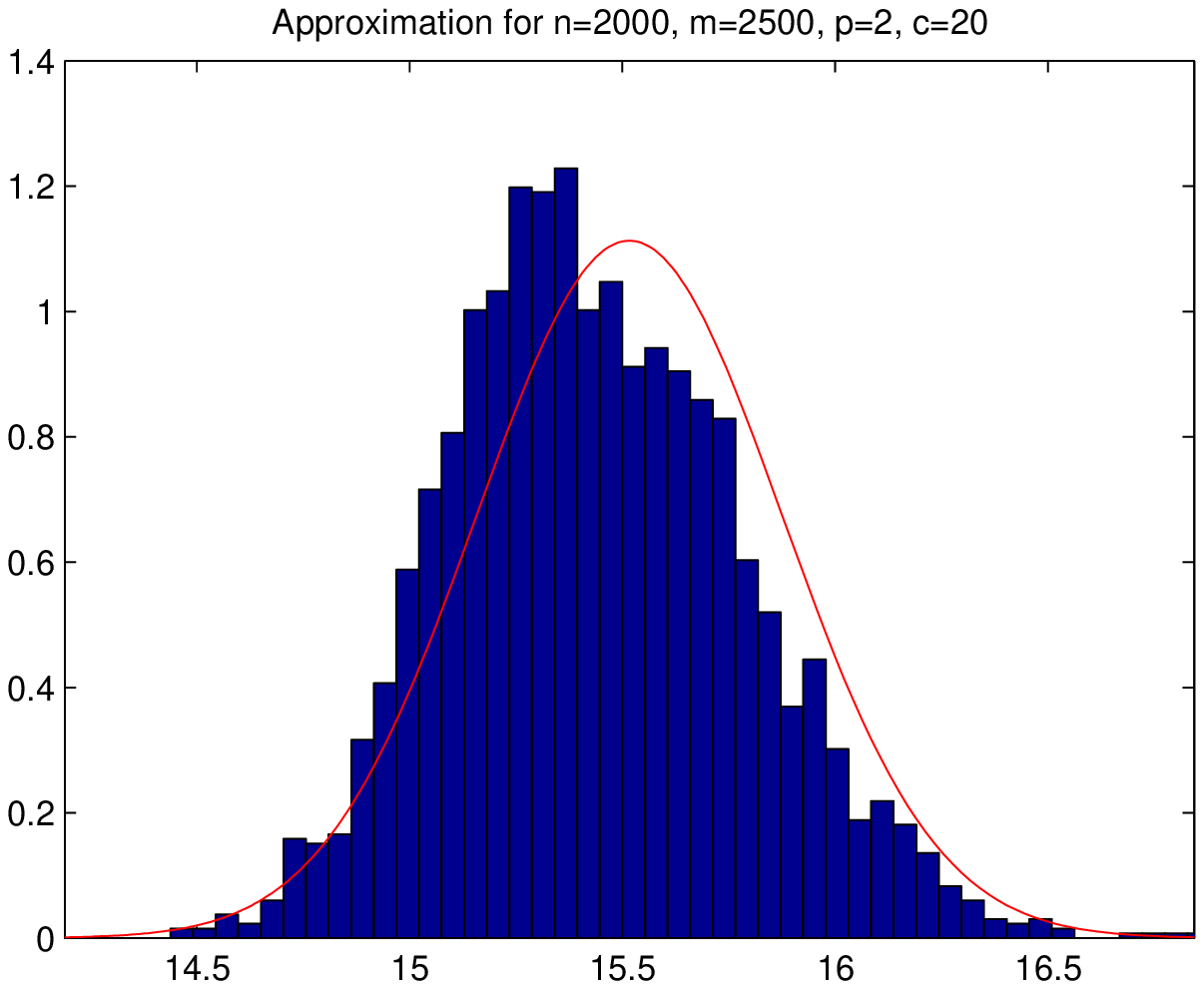}}}
      \subfigure[$p=2.5$]{\scalebox{0.40}{\includegraphics{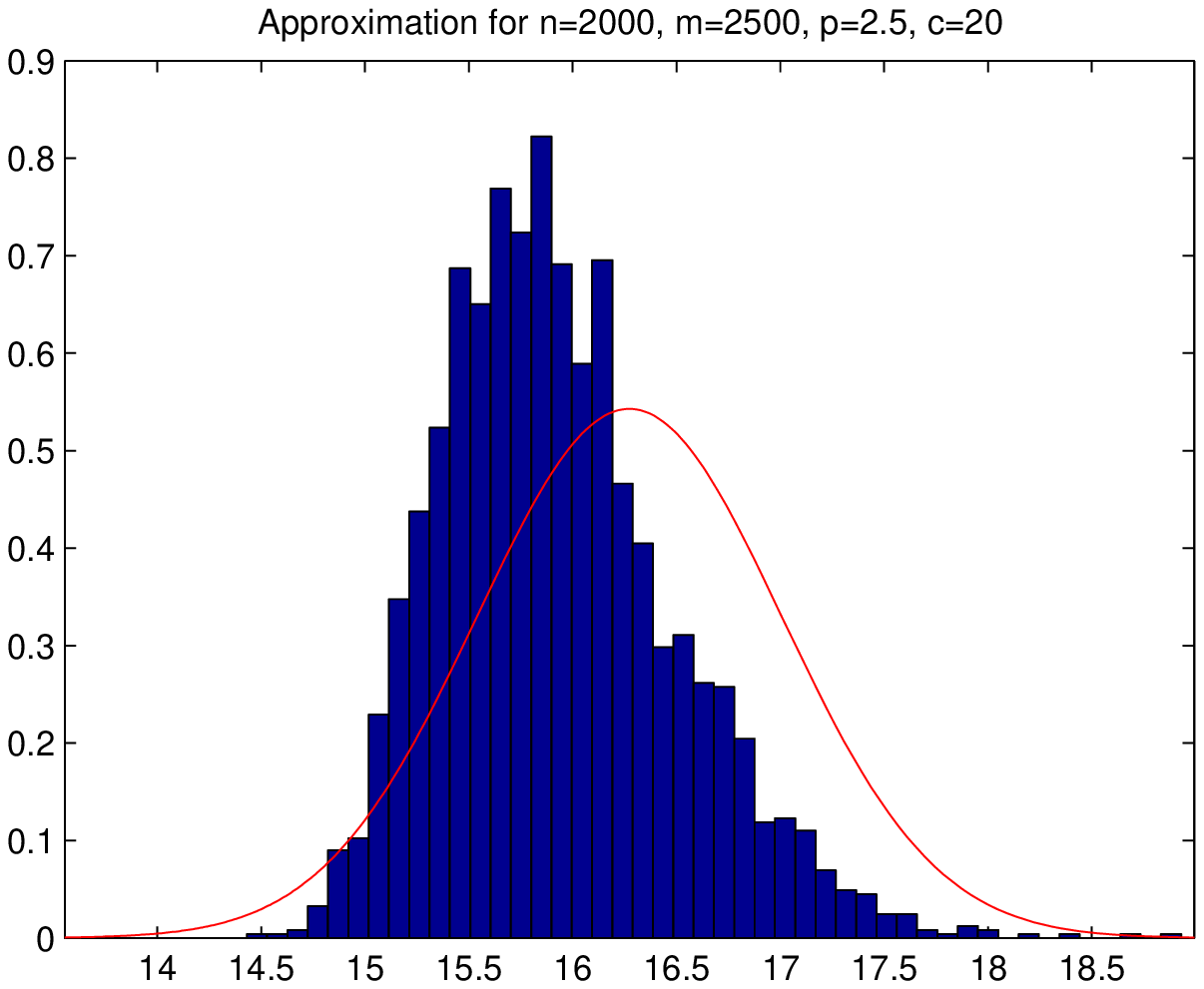}}}
      }
          \caption{\textit{Density histogram of the distribution of the estimator $\rho_n $ for different values of $p$ when $X\sim \mathcal{N}(10,3).$ }}
    \label{fig3}
  \end{center}
\end{figure}
Figure \ref{fig1} shows that the central limit theorem indeed represents a very good approximation which improves significantly with increasing sample size. The small downward
bias that appears in Figure 1 a) is getting increasingly irrelevant with growing sample size. We have experimented with different values of $p$ such as $p=1, 1.5, 2 $ and $2.5$ and we have also changed the value of $\epsilon$ (respectively $c=1/\epsilon$). The tendency shown in Figure \ref{fig1} is largely upheld, however, as expected, the  standard errors are increased when $c$ and/or $p$ is increased. Also, the limiting normal approximation seems to be  more accurate for the same sample sizes when a smaller value of $p$ is used. This discussed effect is illustrated on Figure \ref{fig3} where $p=1$ (i.e., the case of AVaR), $p=1.5,$  $p=2$ (where a different sample in comparison to the sample in Figure \ref{fig1},) and $p=2.5$ was simulated). The remaining quantities have been kept fixed to $n=2000$ and $c=20.$
We stress that increasing the sample size in Figure \ref{fig3} d) makes the histogram look much more like the limiting normal curve so that the discrepancy observed there is indeed just due to the limiting approximation popping up at larger samples when $p$ is increased.

We also experimented with different distributions for the random variable $X.$ We took specifically $t$-distributions with degrees of freedom $\nu$ such as 4, 6,  8 and 60, shifted to have the same mean of 10 like in the normal simulated data. The results of this comparison for $p=2,  \epsilon=0.05$ and $n=4000$ are shown in Figure 2. The variances of the $t$-distributed variables, being equal to $\nu/(\nu -2),$ are finite and even smaller than the variance of the normal random variable in Figure \ref{fig1}. However the heavier tails of the $t$ distribution adversely affect the quality of the approximation. Despite the fact that the limiting distribution of the risk estimator is still normal when $\nu=6$ and $\nu=8,$ the heavy tailed data cause the normal approximation to be relatively poor even at $n=4000.$ The case $\nu=60$ is closer to normal distribution and hence the approximation works better in this case.

Note that the limiting distribution when $p=2$ involves the fourth moment of the $t$ distribution and this moment is finite for $\nu=6, 8$ and $60$ but is infinite when $\nu =4.$ As a result, it can be seen from Figure \ref{fig2} d) that the normal approximation collapses in this case. Also, Figure \ref{fig2} shows that for attaining similar quality in Kolmogorov metric for the asymptotic approximation like in the case of normally distributed $X,$ in Figure 1 c),  much bigger samples are needed. For the fixed sample size of 4000, the quality of the normal approximation worsens as $\nu$ decreases from 60 to 8 and then to 6. Furthermore, and outside of the scope of the present paper, we note that if the distribution of $X$ has even heavier tails than the $t$ distribution with (for example, if it is in the class of stable distributions
 with stability parameter in the range (1,2)) then the limiting distribution of the risk may not be normal at all.
\newpage
\section{Conclusions}
The infinity dimensional delta method is a standard statistical technique to evaluate the asymptotic distribution of estimators of statistical functionals. The applicability of the procedure hinges on veryfing smoothness conditions of the related functionals. Motivated primarily by the need to estimate coherent risk measures we introduce a general composite structure for such functionals in in which all known coherent risk measures can be cast. The potential applicability of our central limit theorems however extends beyond functionals representing coherent risk measures. Our short simulation study indicates that the central limit theorem-type approximations are very accurate when the sample size is large, $p$ is in reasonable limits between 1 and 3 and the distribution of $X$ is with not too heavy tails. We note that for smaller sample sizes, the technique of concentration inequalities may be more powerful and accurate when evaluating the closeness of the approximation. It is possible to derive concentration inequalities for estimators of statistical functionals with the structure that has been introduced in our paper. This is a subject of ongoing research.

\section*{Acknowledgements}

{The first author was partially supported by the NSF grant DMS-1311978.
The second author was partially supported by a research grant PS27205 of The University of New South Wales.
The third author was partially supported by the NSF grant DMS-1312016.}


\end{document}